\documentclass[a4paper,11pt]{article}

\usepackage[scale=0.74]{geometry}

\usepackage{libertine}

\usepackage{amsmath}
\usepackage[parfill]{parskip}
\usepackage{graphicx}
\usepackage{amssymb}
\usepackage{epstopdf}
\usepackage{float}
\usepackage[alphabetic,nobysame,msc-links,initials,short-months]{amsrefs}%
\usepackage{caption}
\usepackage{hyperref}
\usepackage{mathtools}
\usepackage{color}
\usepackage{listings}

\usepackage[normalem]{ulem}
\usepackage{mathrsfs}
\usepackage{bbm}
\usepackage[dvipsnames]{xcolor}
\usepackage{todonotes}
\usepackage{enumitem}
\usepackage{units}

\usepackage{amsthm}
\usepackage{cases}

\newtheorem{theorem}{Theorem}[section]
\newtheorem{corollary}[theorem]{Corollary}
\newtheorem{proposition}[theorem]{Proposition}

\newtheorem{lemma}[theorem]{Lemma}

\newtheorem*{theorem*}{Theorem}
\newtheorem*{proposition*}{Proposition}

\theoremstyle{remark}
\newtheorem{remark}[theorem]{Remark}
\AtBeginEnvironment{remark}{%
  \pushQED{\qed}%
}
\AtEndEnvironment{remark}{\popQED}

\theoremstyle{definition}
\newtheorem{definition}[theorem]{Definition}
\AtBeginEnvironment{definition}{%
  \pushQED{\qed}%
}
\AtEndEnvironment{definition}{\popQED}

\usepackage{ wasysym }
\usepackage{ textcomp }

\newcommand{\de}{\operatorname{d}}

\newcommand{\sG}{\mathscr{G}}

\DeclareMathOperator{\e}{\mathbb{E}}
\DeclareMathOperator{\p}{\mathbb{P}}

\DeclareMathOperator{\T}{\mathbb{T}_d}
\newcommand{\pion}{\partial^{\rm pion}}

\usepackage{orcidlink}
\usepackage[affil-sl]{authblk}

\title{The Second Phase Transition of the Contact Process\\ on a Random Regular Graph}
\author[ ]{John Fernley \orcidlink{0000-0002-6635-4341}}
\affil[ ]{Centre for Research in Statistical Methodology, Department of Statistics,}
\affil[ ]{University of Warwick,
Coventry,
United Kingdom}
\affil[ ]{\href{mailto:john.fernley@warwick.ac.uk}{\rm john.fernley@warwick.ac.uk}}
\date{\today}

\usepackage{framed}
\usepackage{float}
\usepackage[caption = false]{subfig}
\usepackage{graphicx}
\usepackage{makecell}

\usepackage{appendix}
\usepackage[most]{tcolorbox}

\begin{document}

\maketitle

\begin{abstract}
\noindent
The regular tree corresponds to the random regular graph as its local limit. For this reason the famous double phase transition of the contact process on regular tree has been seen to correspond to a phase transition on the large random regular graph, at least at the first critical value.

\noindent
In this article, we find a phase transition on that large finite graph at the second critical value: between linear reinfections and reinfections following a long healthy period.

{\footnotesize
\vspace{1em}

}
\end{abstract}



\section{Introduction}

The simplest model for sparse social connections between $N$ agents is the random regular graph. 
Identifying each sequential $d$ integers in $[dN]=\{1,\dots,dN\}$ as the half-edges of a distinct vertex in $[N]$, when $d\geq 3$ and $dN$ is even, we can generate a uniform perfect matching on $[dN]$ to define the edges of a random regular \emph{configuration model} $\sG$. Each edge is between the vertices that own each of its half-edges.

This model converges to the regular tree $\mathbb{T}_d$ in the weak local sense \cite{remco}, i.e. for any finite radius $r$ and independently uniform random vertex $i\in[N]$ we can compare distributions of balls
\begin{equation}\label{eq_weak}
B_{\sG}(i,r)\stackrel{({\rm d})}{\rightarrow}B_{\mathbb{T}_d}(o,r).
\end{equation}

$o$ on the right hand side of the above equation is the root of the regular tree, but of course the regular tree is translation invariant. Moreover, we can take $r=\epsilon\log N$ for some small $\epsilon>0$ and still find \eqref{eq_weak}. Given $d\geq 3$ we have that $\sG$ is connected with high probability, however the diameter of $\sG$ is $O_{\p}(\log N)$ so of course the correspondence to the regular tree will have broken down on a graph ball of larger logarithmic radius.

On this graph $\sG$ we will put a contact process a.k.a. SIS infection, a simple Markovian model for infection without immunity, where each infected vertex recovers independently at Poisson rate $1$. Also, independently, each undirected edge has a Poisson rate $\lambda>0$ of infection times where the infection is spread. Precisely, at an infection time both vertices of the edge become infected if one of them was before. 
We recall the definition of the weak phase transition \cite{pemantle92}.

\begin{theorem}\label{thm_lambda_1}
For any $d \geq 3$, let $\mathbb{T}_d$ be the $d$-regular tree and $o \in V(\mathbb{T}_d)$ be an arbitrary vertex. $(\xi_t)_t$ is a standard contact process from the point source $\{\xi_0=\{o\}\}$, it hits the empty set at time $\tau_\emptyset$. We find:
\begin{itemize}
\item $\tau_\emptyset=\infty$ with positive probability if $\lambda>\lambda_1(d)$;
\item $\tau_\emptyset<\infty$ with probability $1$ if $\lambda<\lambda_1(d)$.
\end{itemize}
\end{theorem}

This regular tree result has a natural counterpart on the random regular graph. We write the result, to be consistent with the rest of the article, with the standard Landau notation (see Section \ref{sec_notation}). 

\begin{theorem}[ {\cite{valesin16,lalleysu}} ]\label{thm_valesin}
On the $d$-regular configuration model, put the contact process from full initial infection $\xi_0=[N]$ and wait until it hits the empty set at time $\tau_\emptyset$. We find:
\[
\tau_\emptyset=
\begin{cases}
e^{\Theta_{\mathbb{P}}(N)}, & \lambda>\lambda_1(d),\\
\Theta_{\mathbb{P}}({\log N}),  & \lambda<\lambda_1(d).
\end{cases}
\]
\end{theorem}

We recall also the definition of the strong phase transition \cite{pemantle92}.

\begin{theorem}\label{thm_pemantle_transition}
For any $d \geq 3$, let $\mathbb{T}_d$ be the $d$-regular tree and $o \in V(\mathbb{T}_d)$ be an arbitrary vertex. $(\xi_t)_t$ is a standard contact process from the point source $\{\xi_0=\{o\}\}$, and $I_k$ denotes the $k$\textsuperscript{th} time the vertex $o$ becomes reinfected from being healthy. Conditionally on the event that $\xi$ survives, almost surely:
\begin{itemize}
\item $I_k<\infty$ for every positive integer $k$ if $\lambda>\lambda_2(d)$;
\item $I_k=\infty$ for all $k$ large enough if  $\lambda_1(d)<\lambda<\lambda_2(d)$.
\end{itemize}
\end{theorem}

This as yet has no counterpart in finite graphs, which we set out to find in this work. When Theorem \ref{thm_lambda_1} becomes Theorem \ref{thm_valesin} we must translate infinite times into times that are finite and of large order. Similarly we will 
see the phase transition on a finite graph corresponding to Theorem \ref{thm_pemantle_transition} at $I_k$ where $k=k(N)$ grows with the size of the graph.

On a wide range of trees, there has been considerable further attention to the first phase transition. For Galton--Watson trees (with i.i.d. offspring) we see in \cite{durrett20} that $\lambda_1=0$ for any tree where every exponential moment of the offspring distribution is infinite. 
In \cite{bhamidi2019survival} the picture is completed with the result that $\lambda_1>0$ whenever the offspring distribution is \emph{subexponential} in the sense that there is some finite exponential moment. This was explored recently in \cite{immunisation} by looking at bounded versions (with diverging bound) of subexponential offspring distributions still achieving $\lambda_1=0$, with open questions for the lighter tail parameters.

However, there has been no examination of $\lambda_2$ on finite graphs. Note that works \cite{zbMATH01713831,zbMATH06340420} do refer to $\lambda_2(\T)$, but their interest is in a graph where it happens that $\lambda_1$ for the weak local limit (the \emph{canopy tree}, see also \cite[Equation 1.2]{valesin16}) is equal to $\lambda_2(\T)$ and so it is still really a weak survival phase transition that is described. Indeed, this is why the transition of those works at $\lambda_2(\T)$ is the appearance of exponential survival time, as in Theorem \ref{thm_valesin}.

\section{Main results}\label{sec_main}

We find the following feature of the second critical value in finite graphs: that reinfections are linear in the strong survival phase, but in the weak survival phase follow a healthy period asymptotic to $\tfrac{1}{c_\lambda}\log N$. Here the constant $c_\lambda$ is the Malthusian parameter for the contact process \[\frac{1}{t} \log \e_{\, o} \left( |\xi_t| \right) \rightarrow c_\lambda,\] for another definition see Proposition \ref{prop_cp_mean}.

\begin{theorem}\label{thm_main}
For any $d \geq 3$, let $\sG$ be a uniform $d$-regular multigraph on $[N]$. $(\xi_t)_t$ is a standard contact process from the point source $\{\xi_0=\{1\}\}$, and $I_k$ denotes the $k$\textsuperscript{th} time the vertex $\{1\}$ becomes reinfected from being healthy.

We find, conditionally on the event that the infection survives for $\omega(1)$ time,
\begin{numcases}{I_{\left\lfloor\log^{\epsilon} N\right\rfloor}=}
   \Theta_{\mathbb{P}}({\log^{\epsilon} N}), & $\lambda>\lambda_2(d)$, \label{case_strong}
   \\
   \frac{1}{c_\lambda}\log N +\Theta_{\mathbb{P}}({\log^{\epsilon} N}), & $\lambda_1(d)<\lambda<\lambda_2(d)$, \label{case_weak}
\end{numcases}
for any $\epsilon\in (0,1)$.
\end{theorem}

Note that we do not enforce simplicity in the configuration model of our main theorem: we allow both loops and multiply featured edges. This is because as the multigraph version is simple with asymptotically positive probability \cite{janson2009probability} and so the above form is the stronger statement. Also we condition on survival for some time tending to infinity with $N$: of course if we do not have this then $I_{\left\lfloor\log^{\epsilon} N\right\rfloor}=\infty$, but in any event with survival for $\omega(1)$ time we have $I_{\left\lfloor e^{\Omega_{\p}(N)}\right\rfloor}<\infty$.

The main theorem relies on understanding the concentration of the contact process on $\T$. Easily, from Markov's inequality, we can control the right tail using that the expectation is $\Theta(e^{c\lambda})$ \cite{morrowschinazizhang}. However by a new upper bound on the moments of the contact process, in Proposition~\ref{prop_cp_moments}
\begin{equation}\label{eq_moment}
\mathbb{E}_{\, o} \Big( | \xi_t |^n \Big)  = O \left(
t^{n\log_2 
\tfrac{d}{d-2}
}
e^{ n c_\lambda t }
\right),
\end{equation}
we can control the more difficult left tail of a growing infection from $\{\xi_0=o\}$.

\begin{proposition}\label{cor_contact_growth_lower_bound}
The contact process $(\xi_t)_t$ on the regular tree $\T$ with root $o \in V(\mathbb{T}_d)$ has
\[
\p_o\left(
\log | \xi_t | \leq c_\lambda t - 
t^{\delta}
\, \Big| \,
\xi_{t} \neq\emptyset
\right)
\leq e^{-t^{\nicefrac{\delta}{4}}}
\]
for any $\delta\in (0,1)$ and $t$ sufficiently large.
\end{proposition}

To upper bound the left tail, we must stochastically lower bound the infection. Our proof shows the same lower bound on the contact process in a subgraph of $\T$ (see Figure \ref{fig_severed}) and in fact just on the boundary of the infection on that subgraph. This is an improvement on the following result.

\begin{proposition}[ {\cite{lalleysu}} ]
The contact process $(\xi_t)_t$ on the regular tree $\T$ with root $o \in V(\mathbb{T}_d)$ has
\[
\p_o\left(
\log | \xi_t | \leq c_\lambda t - 
\delta t
\, \Big| \,
\xi_t \neq\emptyset
\right)
\rightarrow 0
\]
for any $\delta > 0$ as $t\rightarrow\infty$.
\end{proposition}

This allows us to repeat some estimates with their exponential error factors (of arbitrarily slow base) improved to stretched exponential error factors (of arbitrarily slow stretching exponent), and so get a far better control on the number of contact process paths. Parts of this article therefore are strengthenings of \cite{lalleysu}, in particular the construction of Section \ref{sec_path} is quite similar. However, the algorithmic proof of that section is much shorter than their second moment method approach, and we hope easier to understand. Additionally, the labelling of Section \ref{sec_exploring} is based on their labelling but is somewhat simplified for our purposes, removing their distinction of ``priority'' infected vertices within an infection colouring and removing that a vertex loses its label on recovery. 

\begin{remark}[Locality]
There has been considerable interest since \cite{schramm96} in the conjecture that the critical parameter for bond percolation agrees with the critical parameter of the local limit. Recently in \cite{easo2023critical} this was proved for vertex-transitive graphs.

The contact process itself is essentially a percolation model (where paths must go forwards in time), and the relationship between Theorems \ref{thm_lambda_1} and \ref{thm_valesin} is the same sort of locality. 
On other graphs, \cite{bhamidi2019survival} find on a configuration model with completely general degrees that the critical value is positive for the local limit if and only if it is positive for the finite graph. To be precise, the transition is slightly weaker than that of Theorem \ref{thm_valesin}, but is between extinction at $e^{\Theta_{\mathbb{P}}(N)}$ and  extinction bounded by $N^{1+o_{\mathbb{P}}(1)}$. They naturally conjecture that the nonzero critical values also agree between those other configuration models and their Galton--Watson local limits, but this remains open.

Similarly, in \cite{zbMATH04086722} it was seen that the critical value for the contact process on a long path (to survive exponentially long) agrees with the critical value of the path's local limit: $\mathbb{Z}$.

This work does not add to the literature on the locality of the weak critical value; it could be said that we find the correspondence not of the existence of a path to infinity but instead of infinitely many closed paths through the root. This points to a different type of locality result which could hold in more generality for graphs whose local limit has a double phase transition for the contact process, for instance the same configuration models considered by \cite{bhamidi2019survival}.
\end{remark}

Looking ahead from the results of Section \ref{sec_concentration}, the natural conjecture is that this moment of \eqref{eq_moment} really has no polynomial factor, or even that there is some random variable $W$ such that
\[
e^{ - c_\lambda t }
| \xi_t |
\rightarrow W
\qquad \text{a.s.}
\]
with $W>0$ on the event that $\xi_t$ survives and $\e(W^n)<\infty$ for every positive integer $n$. A result like this might allow Theorem \ref{thm_main} to include reinfections of smaller ordinals $(\omega(1),\log^{o(1)}N$), but seems inaccessible with our timestepping approaches to lower bounding the contact process.

%


\subsection{Notation}\label{sec_notation}


\begin{description}[labelsep=2em, itemsep=-0.3em]
\item[$\T$] The $d$-regular tree.
\item[${[N]}$] The set $\{ 1, \dots, N \}$.
\item[$\ell: \T \rightarrow {[N]}$] The labelling function.
\item[$\sG$] The $d$-regular configuration model on $[N]$.
\item[$\xi_t$]The contact process with $\xi_0=\{o\}$.
\item[$\eta_t$] The \emph{severed} contact process with $\eta_0=\{o\}$, i.e. the contact process on the tree rooted at $o$ where $\de(o)=1$ and all other degrees are $d$.
\item[$\partial\zeta_t$] The boundary subset of some process $\zeta_t$.
\item[$\pion\zeta_t$] The \emph{pioneer} subset of $\partial\zeta_t$: see Definition \ref{def_pioneer}.
\item[$\tilde{\xi}_t, \tilde{\eta}_t$] Coupled infections on $[N]$ mapped by $\ell^{-1}$ into a ``truly infected'' subset of $\xi_t$ or $\eta_t$.
\item[$f(n)=O(g(n))$ or $g(n)=\Omega(f(n))$] $\nicefrac{f(n)}{g(n)}$ is asymptotically bounded by a constant.
\item[$f(n)=o(g(n))$ or $g(n)=\omega(f(n))$] $\nicefrac{f(n)}{g(n)}$ is asymptotically bounded by any positive constant.
\item[$f(n)=\Theta(g(n))$] Both $\nicefrac{f(n)}{g(n)}$ and $\nicefrac{g(n)}{f(n)}$ are asymptotically bounded by a constant.
\item[$f(n)=\Theta_{\p}(g(n))$ etc.] The same statements holding with probability tending to $1$ as $n \rightarrow\infty$.
\end{description}

\section{Concentration of contact process on regular tree}\label{sec_concentration}

First we prove an improvement to existing upper bounds on the first moment.

\begin{proposition}\label{prop_cp_mean}
The contact process $(\xi_t)_t$ has some constant $c_\lambda\geq 0$ such that
\[
e^{c_\lambda t} \leq
\mathbb{E} \Big( | \xi_t | \Big)  \leq
\frac{d}{d-2}
e^{c_\lambda t}.
\]
\end{proposition}

A weaker version of this result is stated as \cite[Theorem 1]{morrowschinazizhang} with upper bound $C(d)e^{c_\lambda t}$, where the arbitrary constant can be extracted from \cite[(3.5)]{madrasschinazi} as $C(d):=\nicefrac{d^2}{d-2}$. Here we count healthy incident branches directly, and so make a simple improvement to their explicit bound on the constant factor as this factor will be significant for our inferred control of the variance.

We could also find a boundary from the Cheeger constant $h(\T)=1-\nicefrac{2}{d}$ if we move to a minimal connected set that contains the infection, but the following proof is a little more intuitive.

\begin{proof}
We first argue that a subset $S$ of size $|S|=k$ vertices in $T_d$ has at least $k(d-2)$ incident free branches:
\begin{itemize}
\item Pick a vertex $v \in S$, which of course alone has $d$ free branches;
\item Add vertices in $S\setminus\{v\}$ in the breadth-first order, these can remove a free branch but they also contribute $d-1$ more free branches in directions away from $v$.
\end{itemize}

Then, stop the infection at time $t$ and split it into $|\xi_t|$ different infection types that continue on the same graphical construction. A vertex $v \in \xi_t$ with $b$ incident $\xi_t$-free branches expects at time $t+s$
\[
\frac{b}{d}\left(
\mathbb{E}\left(
|\xi^{\{v\}}_s|
\right)
-\mathbb{P}\left( v \in \xi^{\{v\}}_s \right)
\right)+\mathbb{P}\left( v \in \xi^{\{v\}}_s \right)
\geq
\frac{b}{d}
\mathbb{E}\left(
|\xi_s|
\right)
\]
vertices infected with its type in its $b$ incident free branches. Hence
\[
\mathbb{E}\left(
|\xi_{t+s}|
\right)
\geq 
\mathbb{E}\left(
|\xi_{t}|
\right)
\cdot
\frac{d-2}{d}
\mathbb{E}\left(
|\xi_{s}|
\right)
\]
which with 
$
\mathbb{E}\left(
|\xi_{t+s}|
\right)
\leq 
\mathbb{E}\left(
|\xi_{t}|
\right)
\mathbb{E}\left(
|\xi_{s}|
\right)
$
 gives the result as in the proof of \cite[Proposition 1]{madrasschinazi}.
\end{proof}

We write down the consequence for the union as it will be required frequently.

\begin{corollary}\label{cor_union}
For the union we have an independent rate $1$ recovery at the first infection of every site contained in the history, excluding the final set which has not recovered, and so
\[
\e\left(
\Big|\bigcup_{s\leq t}\xi_s\Big|
\right)
\leq
\int_0^t
\e\left(
\big|\xi_s\big|
\right)
{\rm d}s
+
\e\left(
\big|\xi_t\big|
\right)
\leq
\frac{d}{d-2}
\left(
\frac{1}{c_\lambda}+1
\right)
e^{c_\lambda t}
\]
whenever $\lambda>\lambda_1$ so that $c_\lambda> 0$ by \cite[Proposition 4.27]{liggett1999stochastic}.
\end{corollary}

%

Now, we can move on to the higher moments.

\begin{proposition}\label{prop_cp_moments}
For any $\lambda>\lambda_1, d\geq 3$ the contact process $(\xi_t)_t$ has
\[
\mathbb{E} \Big( | \xi_t |^n \Big)  =
 O \left(
t^{n\log_2 
\tfrac{d}{d-2}
}
e^{ n c_\lambda t }
\right)
\]
for every $n\geq 2$.
\end{proposition}

\begin{remark}
In \cite[Section 3.1]{schapira2023contact} it is proved that
\[
\left(
\lim_{t \rightarrow \infty}
\mathbb{E} \Big( | \xi_t |^n \Big)^{\nicefrac{1}{t}}
\right)^{\nicefrac{1}{n}}
\]
is non-decreasing in $n$. With the above result, we can say that this sequence is constantly $e^{c_\lambda}$.
\end{remark}


\begin{proof}[Proof of the $n=2$ case of Proposition \ref{prop_cp_moments}]
First, recall that from \cite[Propositions 4.27 and 4.39]{liggett1999stochastic} we have equivalence between $\lambda\geq\lambda_1$ and $c_\lambda\geq 0$. We can upper bound by the Yule process, the Poly\'a urn in which every particle splits at rate $1$: \[\xi_t \preceq Y_{\lambda d t}.\] 

Because the Yule process has a geometric point distribution,
\[
\p(Y_t=k)
=
e^{-t}(1-e^{-t})^{k-1},
\]
we deduce for some initial time $\delta >0$
\begin{equation}\label{eq_yule_variance}
\mathbb{E} \Big( | \xi_\delta |^2 \Big)  \leq
\mathbb{E} \Big( | Y_{\delta \lambda d} |^2 \Big)
=\frac{2-e^{-\delta \lambda d}}{e^{-2\delta \lambda d}}
\leq 2 e^{\delta \lambda d}.
\end{equation}

Write $t=2^k \delta$ where $k=\lceil \log_2 t \rceil$. Then, using the graphical construction of \cite[Chapter III.6]{liggett}, we can upper bound the continuing contact process from time $2^{k-1} \delta$ by independent copies $(\xi^{(1)}_t)_t,(\xi^{(2)}_t)_t,\dots$ of the contact process on regular tree. This gives
\[
\left|\xi_{t}\right|
\preceq
\sum_{i=1}^{\left|\xi_{2^{k-1} \delta}\right|}\left|\xi^{(i)}_{2^{k-1} \delta}\right|.
\]

Expanding the squared sum and using Wald's equation (or equivalently the first line could follow from the usual expression for the Galton--Watson variance at generation $2$ \cite[Chapter I.A.2]{MR2047480}),
\[
\begin{split}
\mathbb{E} \Big( | \xi_t |^2 \Big)  &\leq
\mathbb{E} \Big( | \xi_{2^{k-1} \delta} | \Big) \mathbb{E} \Big( | \xi_{2^{k-1} \delta} |^2 \Big) 
+
\mathbb{E} \Big( | \xi_{2^{k-1} \delta} |^2 
-
 | \xi_{2^{k-1} \delta} | \Big) 
\mathbb{E} \Big( | \xi_{2^{k-1} \delta} | \Big)^2\\
&\leq
\left(
\left(
\frac{d}{d-2}
\right)^2
e^{c_\lambda 2^k \delta}
+
\frac{d}{d-2}
e^{c_\lambda 2^{k-1} \delta}
\right)
\mathbb{E} \Big( | \xi_{2^{k-1} \delta} |^2 \Big) 
-
e^{3c_\lambda 2^{k-1} \delta}\\
&\leq
\left(
1+e^{-c_\lambda 2^{k-1} \delta}
\right)
\left(
\frac{d}{d-2}
\right)^2
e^{c_\lambda 2^k \delta}
\mathbb{E} \Big( | \xi_{2^{k-1} \delta} |^2 \Big) 
\end{split}
\]
requiring in the final line that $c_\lambda\geq 0$. 
Then, iteratively reducing $k$ and recalling \eqref{eq_yule_variance},

\begin{equation}\label{eq_second_moment}
\begin{split}
\mathbb{E} \Big( | \xi_t |^2 \Big) &\leq
\mathbb{E} \Big( | \xi_\delta |^2 \Big)
\prod_{i=1}^{k}
\left(
1+e^{-c_\lambda 2^{i-1} \delta}
\right)
\left(
\frac{d  }{d-2}
\right)^2
\exp\left( c_\lambda 2^i \delta \right)\\
&\leq 
2e^{\delta \lambda d
+
\sum_{i=1}^k e^{-c_\lambda 2^{i-1} \delta}
}
\left( 
\frac{d  }{d-2}
\right)^{2k}
\exp\left( 2 c_\lambda t  \right)\\
&=
O \left(
t^{2\log_2 
\tfrac{d }{d-2}
}
\,
e^{ 2 c_\lambda t }
\right)
\end{split}
\end{equation}
as claimed.
\end{proof}

\begin{proof}[Proof of the $n\geq 3$ case of Proposition \ref{prop_cp_moments}]
We use essentially the same structure as the previous proof. First, the Yule process bound yields
\[
\mathbb{E} \Big( | \xi_\delta |^n \Big)  \leq
\mathbb{E} \Big( | Y_{\delta \lambda d} |^n \Big)
=\sum_{a=1}^\infty a^n e^{-\lambda t} (1-e^{-\lambda t})^{a-1}
<\infty
\]

Write $m_s:=\mathbb{E}\left( \left|\xi_{s}\right|\right)$ for the mean growth, and we have the same time doubling upper bounds 
\[
\left|\xi_{2s}\right|
\preceq
\sum_{i=1}^{\left|\xi_{s}\right|}\left|\xi^{(i)}_{ s}\right|.
\]
for i.i.d. $\xi_s, \xi_s^{(1)}, \xi_s^{(2)} \dots$

Because we already bounded the second moment, we can use Rosenthal's inequality \cite[Theorem 9.1]{gut} for some universal constant $D_n$ to derive
\[
\frac{1}{D_n}
\mathbb{E}\Bigg(
\bigg(
\sum_{i=1}^{\left|\xi_{s}\right|}\left(\left|\xi^{(i)}_{s}\right|-m_s\right)
\bigg)^n
\Bigg| \xi_s
\Bigg)
\leq
\left(
\left|\xi_s\right|
\mathbb{V}{\rm ar}\left|\xi_s\right|
\right)^{\nicefrac{n}{2}}
+\xi_s
\mathbb{E}\left(
\left(
\left|\xi_{s}\right|-m_s
\right)^n
\right)
\]
so by taking expectations and applying \eqref{eq_second_moment}
\[
\begin{split}
\mathbb{E}\left(\Bigg|
\sum_{i=1}^{\left|\xi_{s}\right|}\left|\xi^{(i)}_{ s}\right|
-m_s\left|\xi_{s}\right| \Bigg|^n \right)
&=
\mathbb{E}\left(
\Bigg|
\sum_{i=1}^{\left|\xi_{s}\right|}\left(\left|\xi^{(i)}_{s}\right|-m_s\right)
\Bigg|^n
\right)\\
&\leq
\left(
 m_s^{3+O(\nicefrac{\log s}{s})}
\right)^{\nicefrac{n}{2}}
+
m_s \mathbb{E}\left( D_n
\left(
\left|\xi_{s}\right|-m_s
\right)^n
\right)\\
\end{split}
\]
leading to the moment relation 
\begin{equation}\label{eq_doubling_time}
\begin{split}
\mathbb{E}\left(\left|\xi_{2s}\right|^n \right)^{\nicefrac{1}{n}}
&\leq
\mathbb{E}\left(\Bigg|
\sum_{i=1}^{\left|\xi_{s}\right|}\left|\xi^{(i)}_{ s}\right|
\Bigg|^n \right)^{\nicefrac{1}{n}}
\leq
\mathbb{E}\left(\Bigg|\sum_{i=1}^{\left|\xi_{s}\right|}\left|\xi^{(i)}_{ s}\right|-m_s\left|\xi_{s}\right| \Bigg|^n \right)^{\nicefrac{1}{n}}
+
\mathbb{E}\left(m_s^n\left|\xi_{s}\right|^n \right)^{\nicefrac{1}{n}}\\
&\leq
\left(
 m_s^{3+o(1)}
\right)^{\nicefrac{1}{2}}
+
m_s^{\nicefrac{1}{n}}\mathbb{E}\left( D_n
\left(
\left|\xi_{s}\right|-m_s
\right)^n
\right)^{\nicefrac{1}{n}}
+m_s \mathbb{E}\left(\left|\xi_{s}\right|^n \right)^{\nicefrac{1}{n}}\\
&=
m_s \mathbb{E}\left(\left|\xi_{s}\right|^n \right)^{\nicefrac{1}{n}}
+ m_s^{\nicefrac{3}{2}+o(1)}
.
\end{split}
\end{equation}
Finally, write
\[
\mathbb{E} \Big( | \xi_t |^n \Big)  =
f(t)^n
t^{n\log_2 
\tfrac{d}{d-2}}
e^{ n c_\lambda t }
\]
and by inserting this and Proposition \ref{prop_cp_mean} into \eqref{eq_doubling_time} we see, for $t$ large enough,
$
f(2t)
\leq
f(t)
+
e^{-\tfrac{1}{2} c_\lambda t }
$
which implies $f(t)=O(1)$.
\end{proof}

To lower bound the spread of infection by a Galton--Watson process, the following concept will be useful in providing independent forward infections.

\begin{definition}[Pioneer points]\label{def_pioneer}
A process $(\zeta_t)_t$ has at time $t$ the set of pioneers
\[
\pion \zeta_t :=
\Bigg(
\partial \bigcup_{s\leq t} \zeta_s
\Bigg)
\cap \zeta_t
\]
i.e. points on the frontier of the infection history. In the weakly supercritical contact process on the regular tree it is natural to expect that the majority of infected vertices would be pioneers in this sense.
\end{definition}

Given a set of pioneers, we automatically can lower bound the future infection by a \emph{severed} contact process $\eta_t$ growing from each, i.e. the contact process on the severed tree, where $\de(o)=1$ and all other degrees are $d$ (see Figure~\ref{fig_severed}).

\begin{figure}
\centering
\includegraphics[width=0.66\textwidth]{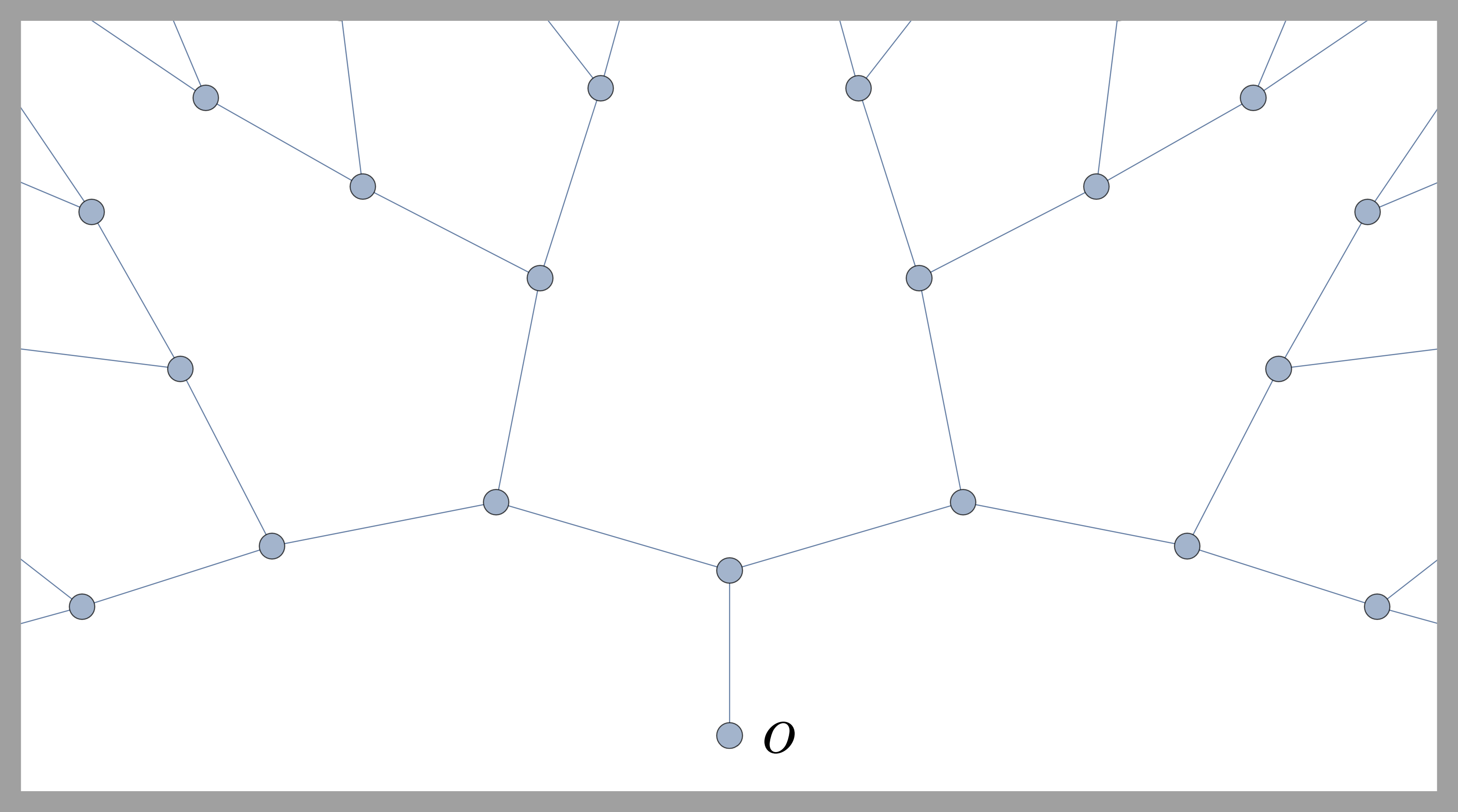}
\caption{A neighbourhood of the root $o$ in the severed tree with $d=3$.}\label{fig_severed}
\end{figure}

Now we can prove a result from Section \ref{sec_main} on the left tail of the infection. Note the symmetrical upper bound follows from the first moment of Proposition \ref{prop_cp_mean} and Markov's inequality; in this problem controlling this lower deviation from the mean is the harder direction.

\begin{proposition}\label{prop_sev_growth_lower_bound}
The severed contact process $(\eta_t)_t$,  from $\eta_0=\{o\}$, has
\[
\p\left(
\log | \eta_t | \leq c_\lambda t - 
t^{\delta}
\, \Big| \,
\eta_t \neq\emptyset
\right)
\leq e^{-t^{\nicefrac{\delta}{4}}}
\]
for any $\delta\in (0,1)$.
\end{proposition}

\begin{proof}
To produce a lower bound, in the first stage we stop the infection when it has exactly $\left\lfloor \tfrac{1}{2}\mathbb{E} \big( | \eta_U | \big) \right\rfloor$ infected vertices, for some large $U$. This is done by constructing a Galton--Watson lower bound with timesteps of length $U$, after each timestep we reduce to see i.i.d. severed processes.

In Proposition \ref{prop_cp_mean} we argued that at least proportion $\left(1-\nicefrac{2}{d}\right)$ of any set in the regular tree must have a free branch, and in the severed case at most one of these branches can contain the root. In this way we have at least
\[
\left\lfloor
\left(1-\frac{2}{d}\right) \eta_U
\right\rfloor
 -1
\]
severed processes to continue the Galton--Watson lower bound. Now the claimed result has conditioning $\{\eta_\infty\neq\emptyset\}$, an event which contains the event of our Galton--Watson lower bound surviving. Moreover because for any positive time we have a positive probability to produce at least $2$ pioneers, that branching structure provides
\[
\p(\eta_\infty\neq\emptyset)
=
\p(\eta_U\neq\emptyset)
+e^{-\Omega(U)}.
\]

This, combined with the inclusion
\[
\{\eta_\infty\neq\emptyset\}
\subset
\{\text{GW survival}\}
\subset
\{\eta_U\neq\emptyset\}
\]
leads to
\begin{equation}\label{eq_conditioning}
\begin{split}
\p\left(
\log | \eta_t | \leq c_\lambda t - 
t^{\delta}
\, \Big| \,
\eta_{\infty} \neq\emptyset
\right)
&=
\frac{
\p\left(
\log | \eta_t | \leq c_\lambda t - 
t^{\delta}
\, ; \,
\text{GW survival}
\right)-e^{-\Omega(U)}}{
\p(\text{GW survival})
-e^{-\Omega(U)}}\\
&\leq
\p\left(
\log | \eta_t | \leq c_\lambda t - 
t^{\delta}
\, \Big| \,
\text{GW survival}
\right)+e^{-\Omega(U)}.
\end{split}
\end{equation}

To hit $\left\lfloor \tfrac{1}{2}\mathbb{E} \big( | \eta_U | \big) \right\rfloor$, we have from \cite[Proposition 2.3]{lalleysu} as $U\rightarrow\infty$
\[
\mathbb{E} \big( | \eta_U | \big)
=\Omega\left(
e^{c_\lambda U}
\right),
\]
we can then apply H\"older's inequality as in \cite[Lemma 4.4]{paleyzygmund} with Proposition \ref{prop_cp_moments} to see
\[
\mathbb{P}\left(| \eta_U | \geq \tfrac{1}{2}  \mathbb{E} \big( | \eta_U | \big) \right)
 \geq
 \frac {\left(\tfrac{1}{2}\mathbb{E} \big( | \eta_U | \big)\right)^{\nicefrac{k}{(k-1)}}}{\mathbb{E} \big( | \eta_U |^k \big)^{\nicefrac{1}{(k-1)}}}
 \geq
 \frac {\left(\tfrac{1}{2}\mathbb{E} \big( | \eta_U | \big)\right)^{\nicefrac{k}{(k-1)}}}{\mathbb{E} \big( | \xi_U |^k \big)^{\nicefrac{1}{(k-1)}}}
 =\Omega \left(
U\right)^{-\tfrac{k}{k-1} \log_2 \tfrac{d}{d-2}}.
\]

Recalling the containment $\{\eta_\infty\neq\emptyset\}
\subset
\{\text{GW survival}\}
\subset
\{\eta_U\neq\emptyset\}$,
\[
\begin{split}
\mathbb{P}\left(| \eta_U | \geq \tfrac{1}{2}  \mathbb{E} \big( | \eta_U | \big) ; \text{GW survival} \right)
&\geq
\left(
\mathbb{P}\left(| \eta_U | \geq \tfrac{1}{2}  \mathbb{E} \big( | \eta_U | \big) \right) - \p(\{\eta_U\neq\emptyset\} \setminus \text{GW survival})
\right)\\
&=\mathbb{P}\left(| \eta_U | \geq \tfrac{1}{2}  \mathbb{E} \big( | \eta_U | \big)  \right)+e^{-\Omega(U)},
\end{split}
\]
and so we have this $\Omega \left(
U\right)^{-\tfrac{k}{k-1} \log_2 \tfrac{d}{d-2}}$ probability in the conditioned Galton--Watson lower bound.

We check if this happens in the first generation, i.e. by time $U$, and otherwise forget the past and restart the argument from some infected vertex (in a Galton--Watson process conditioned to survive, at least one of its offspring is a Galton--Watson process conditioned to survive).

In this way, for any $\epsilon > 0$ we fail to hit $\left\lfloor \tfrac{1}{2}\mathbb{E} \big( | \eta_U | \big) \right\rfloor$ before time
\[U^{1+(1+\epsilon)\log_2 \tfrac{d}{d-2}}\]
with probability bounded by
\begin{equation}\label{eq_hitting}
\begin{split}
\left(
1-\Omega \left(
U\right)^{-\tfrac{k}{k-1} \log_2 \tfrac{d}{d-2}}
\right)^{U^{(1+\epsilon)\log_2 \tfrac{d}{d-2}}}
&\leq
\exp\left(
-
U^{(1+\epsilon)\log_2 \tfrac{d}{d-2}}
\Omega \left(
U\right)^{-\tfrac{k}{k-1} \log_2 \tfrac{d}{d-2}}
\right)\\
&\leq 
\exp\left(-U^{(1-\varphi)\epsilon\log_2 \tfrac{d}{d-2}}\right)\\
\end{split}
\end{equation}
for any small $\varphi>0$ by considering $k$ sufficiently large.

Recall that after reducing this number with the mechanic of the Galton--Watson lower bound, we have at least
\[
\left\lfloor
\left(1-\frac{2}{d}\right) \frac{1}{2}  \mathbb{E} \big( | \eta_U | \big)
\right\rfloor -1=:I
\]
severed processes that can run independently from the initial set that was just found.

The remaining time is $T$,
\[
t-U^{1+(1+\epsilon)\log_2 \tfrac{d}{d-2}}=T,
\]
and the later lower bound is $I$ i.i.d. severed processes $\eta^{a}_T$. The total lower bound at time $t$ is thus
\[
I \mathbb{E} \big( | \eta_T | \big)+
\sum_{a=1}^I
\left(
\eta^{a}_T
-\mathbb{E} \big( | \eta_T | \big)
\right).
\]

It remains to look at concentration of the random term above. By Proposition \ref{prop_cp_moments}
\[
\mathbb{E}\left[\left(
\sum_{a=1}^I
\left(
\eta^{a}_T
-\mathbb{E} \big( | \eta_T | \big)
\right)
\right)^{2}
\right]
\leq
I \mathbb{E} \big( | \eta_T |^2 \big)= O \left( e^{ c_\lambda U }
T^{2\log_2 
\tfrac{d}{d-2}
}
e^{ 2 c_\lambda T }
\right),
\]
that is to say
\[
\mathbb{E}\left[\left(
I \mathbb{E} \big( | \eta_T | \big)-\eta_t
\right)^{2}
\right]
\leq
T^{2\log_2 
\tfrac{d}{d-2}
}
e^{ c_\lambda (2T+U +O(1)) }.
\]

It follows from Markov's inequality

\begin{equation}\label{eq_markov}
\begin{split}
\p\left(
\eta_t\leq
\frac{I}{2}
\mathbb{E} \big( | \eta_T | \big)
\right)
&\leq
\frac{4}{I^{2} \mathbb{E} \big( | \eta_T | \big)^{2}}
T^{2\log_2 
\tfrac{d}{d-2}
}
e^{ c_\lambda  (2T+U +O(1)) }\\
&\leq
T^{2\log_2 
\tfrac{d}{d-2}
}
e^{ c_\lambda  (-U +O(1)) }.
\end{split}
\end{equation}

This event gives 
\[
\eta_t\geq
\frac{I}{2}
\mathbb{E} \big( | \eta_T | \big)
=\Omega\left( 
e^{c_\lambda (U+T)}
\right)
=\Omega\Bigg( 
e^{c_\lambda \Big( t+U- 
U^{1+(1+\epsilon)\log_2 \tfrac{d}{d-2}}\Big)}
\Bigg)
\]
and thus by combining \eqref{eq_markov}, \eqref{eq_hitting} and \eqref{eq_conditioning}
\[
\p\left(
\log | \eta_t | \leq c_\lambda t - 
U^{1+(1+\epsilon)\log_2 \tfrac{d}{d-2}}
\, \bigg| \,
\eta_\infty \neq\emptyset
\right)\leq 
e^{ c_\lambda  (-U +O(\log U)) }
+e^{-U^{(1-\varphi)\epsilon\log_2 \tfrac{d}{d-2}}}
+e^{-\Omega(U)}
\]
and by setting $\epsilon=\frac{1}{\log_2 \tfrac{d}{d-2}}$
\[
\p\left(
\log | \eta_t | \leq c_\lambda t - 
U^{2+\log_2 \tfrac{d}{d-2}}
\, \bigg| \,
\eta_t \neq\emptyset
\right)=e^{-\Omega(t)}+
O\Big(
e^{-U^{(1-\varphi)}}
\Big)=
O\Big(
e^{-U^{(1-\varphi)}}
\Big).
\]

Now because $2+\log_2 \frac{d}{d-2}
<
4$ we can reparametrise to achieve 
\[
\p\left(
\log | \eta_t | \leq c_\lambda t - 
t^{\delta}
\, \Big| \,
\eta_t \neq\emptyset
\right)\ll
e^{-t^{\nicefrac{\delta}{4}}}
\]
or $e^{-t^{\nicefrac{\delta}{(2+o(1))}}}$, if you like, for larger $d$.
\end{proof}

We also need a stronger version of this result applying to the pioneers, which follows from the previous with a proof adapted from \cite[Proposition 2.6]{lalleysu}.

\begin{corollary}\label{cor_pion_growth_lower_bound}
The severed contact process $(\eta_t)_t$,  from $\eta_0=\{o\}$, has
\[
\p\left(
\log | \pion \eta_t | \leq c_\lambda t - 
t^{\delta}
\, \Big| \,
\eta_t \neq\emptyset
\right)
\leq e^{-t^{\nicefrac{\delta}{4}}}
\]
for any $\delta\in (0,1)$.
\end{corollary}

\begin{proof} 
By Corollary \ref{cor_union}, the contact process $\xi_t$ (which stochastically contains the severed contact process) has
\begin{equation}\label{eq_union}
\mathbb{E}\Bigg( \bigg|
\bigcup_{s\leq t} \xi_s
\bigg| \Bigg)
=O\left(
e^{c_\lambda t}
\right)
\end{equation}
and so by Markov's inequality, and by the previous proof for the lower bound,
\begin{equation}\label{eq_growth_condition}
|\eta_t| \geq e^{c_\lambda t - t^{\delta}}
\quad \text{and} \quad
\bigg|
\bigcup_{s\leq t} \xi_s
\bigg|
\leq e^{c_\lambda t + t^{\delta} }
\end{equation}
with probability $1-o(e^{-t^{\nicefrac{\delta}{4}}})$. In Proposition \ref{prop_cp_mean} we found a proportion $1-\nicefrac{2}{d}$ of an arbitrary set on the $d$-regular tree has an empty incident branch. Hence we have a set $S \subset \partial \eta_t$ with
\[
|S|\geq \frac{1}{3}e^{c_\lambda t -  t^{\delta}}
\]
such that every vertex in $S$ has an $\eta_t$-free incident branch. These free branches are disjoint sets, consider how many of them have at least $4 e^{2 t^{\delta}}$ vertices in the union: this cannot be more than
\[
\frac{e^{c_\lambda t + t^{\delta} }}{4e^{2 t^{\epsilon+E_d}}}
\leq
\frac{1}{4}
e^{c_\lambda t - t^{\delta}}.
\]

In this way we have found a subset of $\eta_t$ of size $\tfrac{1}{12}
e^{c_\lambda t - t^{\delta}}$, each having an incident $\eta_t$-free branch with fewer than  $4 e^{2 t^{\delta}}$ vertices in $\bigcup_{s\leq t} \eta_s$.
Each of these vertices has a path of length at most
\[
L=1+\left\lfloor\log_{d-1}\left( 4 e^{2  t^{\delta}} \right)\right\rfloor
\leq
3  t^{\delta}
\]
leading through an incident $\eta_t$-free branch to the boundary of $\bigcup_{s\leq t} \eta_s$. With a simple binomial lower bound on how many of these paths communicate the infection at linear speed, as in \cite[Proposition 2.6]{lalleysu}, we find some $q \in (0,1)$ such that at time $t+L$ at least
\[
\tfrac{1}{5}e^{c_\lambda t - t^{\delta}} q^L
=
e^{c_\lambda t -  \Theta(t^{\delta})}
\]
of these paths produce an element of
\[\eta_{t+L} \cap \Big(\partial \bigcup_{s\leq t+L} \eta_s\Big)=\pion \eta_{t+L},\]
with error probability doubly exponential in $t$ after the conditioning \eqref{eq_growth_condition} (and so insignificant compared to the $o(e^{-t^{\nicefrac{\delta}{4}}})$). By noting that $t+L\sim t$ and rescaling $t^{\delta}$ to absorb a constant factor, we have the claimed lower bound on pioneers.
\end{proof}

Actually, we also need to make an almost trivial change to the conditioning for the following bound.

\begin{proof}[Proof of Proposition \ref{cor_contact_growth_lower_bound}]
The contact process on $\T$ stochastically dominates the severed contact process, and so the event above is contained in the event $\{\log | \pion \eta_t | \leq c_\lambda t - 
t^{\delta}\}$ of the previous corollary. The only difference in the proofs is that we have a weaker conditioning here, and so if our lower bounding severed process dies out we may have to move to one of the other $d-1$ severed trees that make up $\T$ to find a surviving infection to restart.

Eventually though with the same proof we find some subtree containing the same $\pion \eta_t$ context. 
\end{proof}

Finally in this section, we control the separation of two independent contact processes on the same tree. Later, this will allow us to explore the same tree twice as if mostly independent.


\begin{proposition}\label{prop_union_intersection}
When $\lambda_1<\lambda<\lambda_2$, two i.i.d. contact processes $\xi_t^{(1)}$ and $\xi_t^{(2)}$ have
\[
\e\left(
\Big|\bigcup_{s\leq t}\xi_s^{(1)}\cap\bigcup_{s\leq t}\xi_s^{(2)}\Big|
\right)
\leq 
2 \left(2+\frac{1}{c_\lambda}\right)
e^{\tfrac{c_\lambda}{2} t}.
\]
\end{proposition}

\begin{proof}
If we write $e_n$ to denote a vertex at distance $n$ from the root $o$, one contact process has expectation
\[
\e\left(
\Big|\bigcup_{s\leq t}\xi_s^{(1)}\Big|
\right)=
1+\sum_{n=1}^\infty d(d-1)^{n-1}
 \p\left(
e_n \in \bigcup_{s\leq t}\xi_s^{(1)}
\right)
\]
so of course two independent processes have expectation
\[
\begin{split}
\e\left(
\Big|\bigcup_{s\leq t}\xi_s^{(1)}\cap\bigcup_{s\leq t}\xi_s^{(2)}\Big|
\right)&=1+\sum_{n=1}^\infty d(d-1)^{n-1}
 \p\left(
e_n \in \bigcup_{s\leq t}\xi_s^{(1)}\cap\bigcup_{s\leq t}\xi_s^{(2)}
\right)\\
&=1+\sum_{n=1}^\infty d(d-1)^{n-1}
 \p\left(
e_n \in \bigcup_{s\leq t}\xi_s^{(1)}
\right)^2.
\end{split}
\]

Hence by Cauchy--Schwarz
\[
\begin{split}
\e\left(
\Big|\bigcup_{s\leq t}\xi_s^{(1)}\cap\bigcup_{s\leq t}\xi_s^{(2)}\Big|
\right)^2
&\leq
\e\left(
\Big|\bigcup_{s\leq t}\xi_s^{(1)}\Big|
\right)
\left(1+
\sum_{n=1}^\infty d(d-1)^{n-1}
 \p\left(
e_n \in \xi_t^{(1)}
\right)^3
\right)\\
&\leq
\e\left(
\Big|\bigcup_{s\leq t}\xi_s^{(1)}\Big|
\right)
\left(1+
\frac{d}{(d-1)(\sqrt{d-1}-1)}
\right)
\end{split}
\]
the last line using from \cite[Equation 4.49 and Theorem 4.65]{liggett1999stochastic} that ${\p\left(
e_n \in \bigcup_{s\leq t}\xi_s^{(1)}
\right)
\leq (d-1)^{-\nicefrac{n}{2}}}$ whenever $\lambda<\lambda_2$. Finally we recall Corollary \ref{cor_union} for the conclusion
\[
\begin{split}
\e\left(
\Big|\bigcup_{s\leq t}\xi_s^{(1)}\cap\bigcup_{s\leq t}\xi_s^{(2)}\Big|
\right)
&\leq 
\sqrt{\left(1+
\frac{d}{(d-1)(\sqrt{d-1}-1)}
\right) \cdot\e\left(
\Big|\bigcup_{s\leq t}\xi_s^{(1)}\Big|
\right)}\\
&\leq
\sqrt{\left(1+
\frac{d}{(d-1)(\sqrt{d-1}-1)}
\right) 
\frac{d}{d-2}
\left(
\frac{1}{c_\lambda}+1
\right)
e^{c_\lambda t}
}\\
&\leq
4
e^{\tfrac{c_\lambda}{2} t} \sqrt{1+\frac{1}{c_\lambda}}
\leq
2 \left(2+\frac{1}{c_\lambda}\right)
e^{\tfrac{c_\lambda}{2} t}
\end{split}
\]
by in the final line bounding the factor at $d=3$ where it is largest.
\end{proof}


\section{Reinfection on the configuration model}

Our real interest is not in the regular tree but in regular random graphs, and in this section we control upper and lower bounds on the reinfection times.

\subsection{Exploring the configuration on the regular tree}\label{sec_exploring}

In this section we describe how the tree $\T$ is labelled as it is infected, which couples the infection on $\T$ to the infection on the regular configuration model. Precisely, we keep track of a subset of the \emph{truly infected} vertices in $\tilde{\xi}$, on which the random function $\ell$ is bijective to the infection on $[N]$. This construction is similar to that of \cite[Section 3]{lalleysu}, but considerably simplified. It is also similar to that of \cite{valesin16}, a work which also builds the infection on a graph from the infection on the universal cover of that graph---they elect to prevent extra infection in the cover rather than label it as false, but the difference is not significant.

\begin{definition}[$\ell$]\label{def_ell}
$\T$ has one vertex labelled $o$ and called the root, from which we run a contact process with $\tilde{\xi}_0=\{o\}$. We have initial labelling
\[
\ell(o)=1
\]
and no other vertex is labelled. Moreover we keep track of which vertices in $\tilde{\xi}$ are \emph{truly infected}, and $o$ is truly infected at this time.

The map $\ell:\T \rightarrow [N]$ is increasingly determined by the random graph, which we explore with the infection through its history. The graph $\sG$ is encoded by $\nicefrac{dN}{2}$ pairs of half-edges, $d$ from each vertex, which we discover as \emph{the configuration}. 

Every time a vertex in the true infection infects a vertex with no label in $\ell$, a free half-edge from the source vertex is paired with a uniformly selected other free half-edge and this pair is added to the configuration. The associated vertex of the second half-edge is then fixed as the label of the target under $\ell$, and the target is truly infected.

Whenever a vertex in $\T$ is newly labelled by $\ell$ and infected, its neighbourhood is immediately given all other neighbours already \emph{determined} in the configuration but not yet reflected in $\ell$ on its neighbourhood, in uniform random positions.

When a vertex $v$ with a label $\ell(v)=i$ is infected by some $w\in\tilde{\xi}$, we have the following rules:
\begin{itemize}
\item if $w$ was truly infected:
\begin{itemize}
\item if there is another truly infected vertex with label $i$, the target $v$ is now falsely infected;
\item otherwise the target becomes truly infected (even if it was falsely infected before);
\end{itemize}
\item if $w$ was falsely infected:
\begin{itemize}
\item if $v$ is already truly infected then it continues to be;
\item otherwise $v$ becomes falsely infected.
\end{itemize}
\end{itemize}

We can write $\ell^{-1}(\{\text{truly infected vertices}\})$ for the infection on $\sG$. The point of this construction is to allow the contact process on $\T$ to run unmodified, while enforcing on a subset that there is at most one ``true'' position of any label in $[N]$ on the tree at any time, so that we do not double-count the rate of infection or recovery at any vertex to break the coupling.
\end{definition}

Once we have determined the full configuration we can if desired complete the function $\ell$ on all of $\T$, and this labelling gives the covering map between the finite graph $\sG$ and its \emph{universal cover} $\T$. This is the unique simply connected covering space \cite[Chapter 1]{hatcher2002algebraic} and it covers every other connected cover of the space. For finite graphs it must be a tree, for non-tree graphs it must be an infinite tree, and for regular configuration models it must deterministically be the regular tree.

Note also that we didn't enforce that the explored configuration model is a simple graph: it can have both loops and multiple edges. This produces the stronger form of our high probability main theorem, as the multigraph configuration produces the simple configuration model on an event with asymptotically positive probability (asymptotic to $e^{\nicefrac{(1-d^2)}{4}}$).

\begin{remark}\label{rem_synchronicity}
The power of this construction is that we can arbitrarily reduce the truly infected set to still find a path in the space-time percolation structure of the contact process on $\sG$. In this way, we can run the evolution for different parts of the true infection without synchronicity: for example in Section \ref{sec_path} we will explore small time severed contact processes from a set of truly infected pioneers. Each pioneer can produce a truly infected set while its labels are unique, and if a later evolution interferes with that uniqueness, we just say the later pioneer is no longer truly infected.
\end{remark}

We comment also that the construction is slightly strange at multigraph features, but it does not create issues. While $k$ copies of a healthy neighbour are all healthy, one of them becomes truly infected at rate $k\lambda$. After one of them is truly infected, any infections to the other copies will not be true infection and so can be ignored. Similarly, a loop at a truly infected vertex can lead to it infecting copies of itself but these will also not be true infection.

In the remainder of this work, we use this construction to produce the upper and lower bounds implicit in the cases of Theorem \ref{thm_main}. Note that an upper bound on the reinfection times requires a lower bound on the infection, and vice versa.

\subsection{Lower Bound}

First we remark for the \eqref{case_strong} lower bound that we recover at most at Poisson rate $1$ and so trivially $I_k=\Omega_{\mathbb{P}}(k)$. The other lower bound is not totally trivial but is not much more difficult.


\begin{proof}[Proof of \eqref{case_weak} lower bound]
Suppose $k\rightarrow \infty$ and $\lambda_1(d)<\lambda<\lambda_2(d)$. 
In the infection on the tree, the root is reinfected fewer than $\tfrac{k}{3}$ times with high probability because we are in the weak survival phase.

Note that these labellings are not independent, but every other vertex in the tree has a probability $\nicefrac{1}{N}$ to be mapped to the root in the configuration model, except the children of the root where this probability is \[\frac{d-1}{dN-1}\leq \frac{1}{N}. \]

So we can control the first moment: at time $t=\tfrac{1}{c_\lambda}\log N-\tfrac{k}{3}$, recalling Corollary \ref{cor_union}, we expect to infect other occurrences of the root for a total time at most
\[
\frac{1}{N}
\e\left(
\Big|\bigcup_{s\leq t}\xi_s
\setminus \{o\}
\Big|
\right)
=o\left(
1
\right)
\]
and so with high probability we never infect it elsewhere in the tree. Moving then to the infection on the finite graph, the remaining $\tfrac{2k}{3}$ infections will need a further time at least $\tfrac{k}{2}$ by the concentration of the rate $1$ Poisson process of recoveries at the root, and so with high probability 
\[
I_k \geq \frac{1}{c_\lambda}\log N +\tfrac{k}{6}=\frac{1}{c_\lambda}\log N +\Omega_{\mathbb{P}}(k)
\]
as claimed for $k=\left\lfloor\log^{\epsilon} N\right\rfloor$.
\end{proof}

\subsection{Upper Bound}

For the main construction of this section we will put a large set of contact processes as different \emph{types} on the same graphical construction---note that infections sharing a vertex will have totally coalesced until they recover at the same time. First, we control that linearly many types survive.

\begin{lemma}\label{lem_surviving_types}
For each $i \in [k]$ we have type $k$ contact process
\[\left(\xi_t^{(i)}\right)_{t\geq i} \text{\quad with initial condition } \xi_i^{(i)}=\{o\}.
\]

Count the number of surviving types as
\[
S=\sum_{i=1}^k \mathbbm{1}_{\xi^{(i)}\text{ survives}},
\]
then for any small $\epsilon>0$ we find $S=k p_\lambda +o_{\mathbb{P}}(k)$.
\end{lemma}

\begin{proof}
 These contact processes exist on the same graphical construction, but of course times before and after time $j$ operate independently. Since $\mathbb{P}(\xi_t \neq \mathbf{0})\rightarrow p_\lambda$, as $j-i\rightarrow\infty$
 
%

\[
\p \left( 
\text{$\xi^{(i)}$ and $\xi^{(j)}$ both survive}
 \right)
 \leq
 \p \left( 
\text{$\xi^{(i)}$ survives to time $j$}
 \right)
 \p \left( 
\text{$\xi^{(j)}$ survives}
 \right)
\rightarrow
 p_\lambda^2.
\]

The number of surviving types $S$ has 
$
\e(S)=k p_\lambda
$, and
\[
\e(S^2)\leq
k p_\lambda
+2\sum_{1\leq i<j \leq k} \p \left( 
\text{$\xi^{(i)}$ and $\xi^{(j)}$ both survive}
 \right)
=
k^2 p_\lambda^2
\left(1+o(1)\right)
\]
so that $\mathbb{V}{\rm ar}(S)=o(k^2)$. We conclude by Chebyshev's inequality.
\end{proof}

The construction of Section \ref{sec_exploring} will only see true infection for an initial phase of the contact process exploring $\sG$. In the following lemma we control how long this phase is.

\begin{lemma}\label{lem_blue_tree}
For any $\delta\in(0,1)$ a contact process exploring $\T$ can be exactly coupled to a contact process on the network for time \[t=\frac{1}{2 c_\lambda}\log N
- \log^\delta N\] with probability $1-\nicefrac{1}{\log N}$.
\end{lemma}

\begin{proof}
When $\left|\cup_{s\leq t} \, \tilde{\xi}_t \right|=i-1$ and we have seen no false infection yet, we must have used $2i-4$ half edges of the configuration. As there were no repeated labels so far, the next explored half-edge is a clash with exactly probability
\[
\frac{d+(d-1)(i-1)-1}{dN-2i+4-1}
\]
where we subtracted $1$ from numerator and denominator because the half-edge cannot pair with itself.

Iteratively we saw no clashes in $1\ll k \leq \sqrt{N}$ vertices with probability
\[
\prod_{i=2}^k
\left(
1-\frac{(d-1)(i-2)}{dN-2i+3}
\right)
\geq
\prod_{i=2}^k
\left(
1-\frac{i}{N}
\right)
\geq
1-\frac{k^2}{N}.
\]

To use this, we need the second moment of the union. Note because each vertex recovers independently at rate $1$
\[
{\rm Bin}\left(
\Bigg|\bigcup_{s\leq t}\xi_s \Bigg|
,
1-\frac{1}{e}
\right)
\preceq
\int_0^{t+1}
\xi_s
{\rm d}s
\]
which implies by taking second moments
\[
\frac{1}{e}\left(1-\frac{1}{e}\right)
\e\left(
\Bigg|\bigcup_{s\leq t}\xi_s \Bigg|
\right)
+
\left(1-\frac{1}{e}\right)^2
\e\left(
\Bigg|\bigcup_{s\leq t}\xi_s \Bigg|^2
\right)
\leq
\e\left(
\int_0^{t+1}
\int_0^{t+1}
\xi_s
\xi_u
{\rm d}s
{\rm d}u
\right).
\]

For the right-hand side
\[
\e\left(
\int_0^{t+1}
\int_0^{t+1}
\xi_s
\xi_u
{\rm d}s
{\rm d}u
\right)
\leq
(t+1)^2
\e\left(
\xi_{t+1}^2
\right)
=
O \left(
t^{2+2\log_2 
\tfrac{d }{d-2}
}
\,
e^{ 2 c_\lambda t }
\right)
\]
and so we conclude
\[
\p\left(
\text{clash by time }t
\right)\leq
\frac{O \left(
(t+1)^{2+2\log_2 
\tfrac{d }{d-2}
}
\,
e^{ 2 c_\lambda (t+1) }
\right)}{\left(1-\frac{1}{e}\right)^2 N}
=
O \left(
\frac{
t^{2+2\log_2 
\tfrac{d }{d-2}
}
\,
e^{ 2 c_\lambda t }
}{N}\right).
\]

Thus at time $t=\frac{1}{2 c_\lambda}\log N
- \log^\delta N$
\[
\p\left(
\text{clash by time }t
\right)\leq
O \left(
\left(\frac{1}{2 c_\lambda}\log N\right)^{2+2\log_2 
\tfrac{d }{d-2}
}
\,
e^{ -2 c_\lambda \log^\delta N }
\right)
\leq
\frac{1}{\log N}
\]
because a stretched exponential is larger than a polynomial.
\end{proof}

Note at this point that the previous lemma trivialises the $O_{\p}(k)$ upper bound of Theorem \ref{thm_main}.

\begin{proof}[Proof of \eqref{case_strong} upper bound]
For parameters $\lambda>\lambda_2$, in Theorem \ref{thm_pemantle_transition} we see that these reinfections are certain.
Conditioning on survival is the same as conditioning on infinitely many reinfections, and in $\T$ there is no $N$, so they must occur at linear speed. We conclude
\[
I_k=O_{\p}(k)
\]
for any $k=o(\log N)$ by Lemma \ref{lem_blue_tree}.
\end{proof}

\subsubsection{Finding the infection path}\label{sec_path}

\begin{remark}
In this section we have the following essential proof strategy. 

\begin{itemize}[leftmargin=*]
\item By duality we reverse time of the second half of the reinfection journey, and so: 
infect the root $1 \in V(\sG)$ with type $i \in [k]$ at time $i$, up to some final time $k \in \mathbb{N}$. These infections exist on the same graphical construction and so will coalesce.
\item On the same tree but on an independent graphical construction, we run the first part of the infection history which is a single type. This stops at some time $t_1 \approx t_2$. While we still have $o(\sqrt{N})$ vertices these trees will only see true infection, but we must go a little beyond that.
\item Most surviving reversed infections will intersect with the one forwards infection that has been conditioned to survive, and thus we have found reinfections between times $t_1+t_2-k$ and $t_1+t_2$.
\end{itemize}
\end{remark}

For the infection began at time $i \in [k]$, we need a total time larger than $\frac{1}{c_\lambda}\log N $ to produce a large number of meetings
\[
t_+:=
\frac{\log N + \sqrt{k\log N}}{c_\lambda}
 + i
\]
and both halves of the infection initially grow for a time where all labels are distinct on their individual explored trees w.p. $1-o(1/k)$ (note however that they are exploring the same tree)
\[
t_1:=
\frac{\tfrac{1}{2}\log N - \sqrt{k\log N}}{c_\lambda}
\]
but then the reverse infection must be grown for additional time
\[
\Delta t:=\frac{3}{c_\lambda} \sqrt{k\log N}
\]
to ensure that we have found labels of the other infection, at total time (including the initial waiting period $i$)
\[
t_2:=\frac{\tfrac{1}{2}\log N + 2\sqrt{k\log N}}{c_\lambda}
 + i.
\]

In the end we will control types from the set of $k=\lfloor \log^\epsilon N \rfloor$ infections by a simple union bound, so in the majority of the rest of the proof consider just the forwards infection $\xi$ and one arbitrarily chosen reversed infection $\xi'$ from the surviving set of Lemma \ref{lem_surviving_types}. For this purpose, we need $1-o(\nicefrac{1}{k})$ control of various high probability assumptions, which we will refer to as holding ``with sufficiently high probability'' throughout.

\begin{proof}[Proof of \eqref{case_weak} upper bound]
$\xi$ and $\xi'$ are two independent contact processes on the $d$-regular tree, and for both the root $o$ is projected to $1 \in [N]$. Moreover these explore the same graph and so they share the same $\ell$ of Definition \ref{def_ell}.

By Markov's inequality and Corollary \ref{cor_union} we can pick a large constant $C>0$ such that
\[
\mathbb{P}\left(
\log \left(
\Bigg|
\bigcup_{s\leq t_1}
\tilde{\xi}_{s}
\Bigg|
\wedge
\Bigg|
\bigcup_{s\leq t_1}
\tilde{\xi}_{s}'
\Bigg|
\right)
\leq 
c_\lambda t_1 + 
C \log \log N \Bigg|
\xi \text{ and } \xi' \text{ survive}
\right) \geq 
1-\frac{1}{\log N}
\]
and at this point by Lemma \ref{lem_blue_tree} every discovered label has been given a unique label by $\ell$, with sufficiently high probability.

Directly from Proposition \ref{cor_contact_growth_lower_bound} we have with sufficiently high probability
\begin{equation}\label{eq_pioneers}
\log \left(
\big|
\pion\tilde{\xi}_{t_1}
\big|
\wedge
\big|
\pion\tilde{\xi}_{t_1}'
\big|
\right)
\geq 
c_\lambda t_1 - 
t_1^{\epsilon}
\end{equation}
and we will finish the path by connecting these pioneers.
First it's easier to say they have distinct labels, and for that Proposition \ref{prop_union_intersection} controls
\[
\log \left(
\Bigg|
\bigcup_{s\leq t_1}
\tilde{\xi}_{s}
\cap
\bigcup_{s\leq t_1}
\tilde{\xi}_{s}'
\Bigg|
\right)
\leq 
\frac{c_\lambda}{2} t_1 + 
C \log \log N.
\]

These unions are necessarily connected sets and so if a pioneer of one is not a pioneer of the other it must be in this set, and discarding a set of size $\sqrt{e^{c_\lambda t}}$ has no tangible effect on the pioneers we controlled in $F_1$.

One by one from each of the remaining pioneers in $\xi'$, construct a further severed exploration $\eta^{(i)}$ for time $\Delta t -2$. The path we look to construct uses the event:
\begin{enumerate}[label={[\arabic*]}]
\item $\eta^{(i)}$ labels are distinct from each other and all previous labels;
\item\label{item_target} in unit time, $j$ infected only one vertex, neither $j$ nor that vertex attempted recovery, and that vertex adjacent to $j$ was labelled $\ell(j)=X$ distinct from all previous $\xi$ labels;
\item\label{item_source} some pioneer of $\eta^{(i)}$ in unit time (after the additional time $\Delta t -2$) also infected only one vertex, neither recovered, and the vertex was also labelled $X$, distinct from all previous $\xi'$ labels.
\end{enumerate}

Looking at the timeframe of the full construction, two contact processes for times $t_1$ and $t_2$ only expect to infect vertices numbering
\[
O(\sqrt{N})\left(
e^{-c_\lambda \sqrt{k\log N}}
+
e^{c_\lambda k + 2\sqrt{k\log N}}
\right)
=
\sqrt{N}
e^{2\sqrt{k\log N}+O(k)}
\]
so we can say with sufficiently high probability they will not infect more than
\[
\sqrt{N}
e^{3\sqrt{k\log N}}
\]
and hence at any point in the exploration, a sampled half-edge belongs to a new vertex independently with probability at least
\[
1-
\frac{1}{\sqrt{N}}
e^{4\sqrt{k\log N}}.
\]

Separately, an individual severed exploration $\eta^{(i)}$ has expected size less than $O(e^{3\sqrt{k\log N}})$ and so will never exceed $e^{4\sqrt{k\log N}}$. Thus the exploration at $i$ produces a repeated label, by Markov's inequality again, with probability bounded by
\[
\frac{1}{\sqrt{N}}
e^{8\sqrt{k\log N}}.
\]

Now consider all the targets of item \ref{item_target} together. From \eqref{eq_pioneers} we have at least $e^{c_\lambda t_1 - t_1^{\epsilon}}$ of them, and we ask for a positive probability event in unit time. So, with sufficiently high probability,
\begin{equation}\label{eq_xj}
|X_J|=
\Omega\left(
e^{c_\lambda t_1 - 
t_1^{\epsilon}}
\right)
\end{equation}
of these forwards infection pioneers $j$ infect a single neighbour and it is given some label distinct from all other observed labels and recorded in $X_J$.

In the other direction, everything except the severed infections of $\pion\tilde{\xi}_{t_1}$ is discarded: we only look for infection paths on the finite graph that follow this route, and if we ignore infection elsewhere it cannot cause an issue by the monotonicity of the contact process. Moreover we are free to put labels of $\ell$ in any convenient order: any possible spanning tree is a coupling of the infection. This separation also makes it clear that we can explore each $\eta^{(i)}$ one-by-one and do not need to worry about the simultaneous timeframes of different infections (see also Remark \ref{rem_synchronicity}).

In this way we see \emph{i.i.d.} severed explorations $\eta^{(i)}$, and by Proposition \ref{prop_sev_growth_lower_bound}
\[
\p\left(
| \pion \eta^{(i)}_{\Delta t-2} | \leq e^{\tfrac{5}{2}\sqrt{k\log N} }
\, \bigg| \,
\eta^{(i)}_{\Delta t-2} \neq\emptyset
\right)
\leq e^{-(k\log N)^{\Omega(1)}}.
\]

For each severed exploration, we have a positive probability event where the whole set $\eta^{(i)}_{\Delta t-2}$ is truly infected, and a positive proportion $\Omega(e^{\tfrac{5}{2}\sqrt{k\log N} })$ of these pioneers infect a single new vertex each. Thus looking at all the $i$ together with each gives a binomial number of these events, with sufficiently high probability the number of valid labels of item \ref{item_source} is therefore
\begin{equation}\label{eq_xi}
|X_I|=\Omega(e^{\tfrac{5}{2}\sqrt{k\log N} }) \cdot \Omega\left(
e^{c_\lambda t_1 - 
t_1^{\epsilon}}
\right).
\end{equation}

We have constructed  a sufficiently high probability event where the two sets $X_I$ and $X_J$ of labels have sizes \eqref{eq_xi} and \eqref{eq_xj}. Because these draws are approximately uniform (if desired they could be thinned to a set of the same order that was drawn with equal probability), we fail conditionally to see a shared label with the hypergeometric probability
\[
\exp\left(
-\Omega
\left(
\frac{|X_I|\cdot |X_J|}{N}
\right)
\right)
=
e^{
-
e^{\frac{1}{2}\sqrt{k \log N}-O(k)}
}
=o\left(\frac{1}{k}\right)
\]
and so with sufficiently high probability there is a shared label.
\end{proof}

\section*{Acknowledgements}

Thanks to Emmanuel Jacob for many useful discussions.

\bibliography{bib}

\end{document}